\documentclass[10pt,reqno]{amsart}

\usepackage[T1]{fontenc}
\usepackage[utf8]{inputenc}
\usepackage{microtype}
\usepackage{mathtools,amssymb,amsthm}
\usepackage[mathlines]{lineno}
\usepackage{csquotes}
\usepackage[
  backend=biber,
  style=numeric,
  sorting=none,
  maxbibnames=99
]{biblatex}
\usepackage[colorlinks=true,linkcolor=blue,citecolor=blue,urlcolor=blue]{hyperref}
\usepackage[nameinlink,noabbrev]{cleveref}

\allowdisplaybreaks
\numberwithin{equation}{section}

\newcommand{\R}{\mathbb{R}}
\newcommand{\C}{\mathbb{C}}
\newcommand{\Mat}{{M}}
\newcommand{\GL}{{GL}}
\newcommand{\U}{{U}}
\newcommand{\SU}{{SU}}

\newcommand{\Herm}{\operatorname{Herm}}
\newcommand{\Sym}{\operatorname{Sym}}
\newcommand{\Diag}{\operatorname{Diag}}
\newcommand{\Tr}{\operatorname{Tr}}
\newcommand{\Id}{{I}}
\newcommand{\dd}{{d}}

\theoremstyle{plain}
\newtheorem{theorem}{Theorem}[section]
\newtheorem{proposition}[theorem]{Proposition}
\newtheorem{lemma}[theorem]{Lemma}
\newtheorem{corollary}[theorem]{Corollary}

\theoremstyle{definition}

\theoremstyle{remark}
\newtheorem{remark}[theorem]{Remark}

\crefname{theorem}{theorem}{theorems}
\Crefname{theorem}{Theorem}{Theorems}
\crefname{proposition}{proposition}{propositions}
\Crefname{proposition}{Proposition}{Propositions}
\crefname{lemma}{lemma}{lemmas}
\Crefname{lemma}{Lemma}{Lemmas}
\crefname{corollary}{corollary}{corollaries}
\Crefname{corollary}{Corollary}{Corollaries}
\crefname{definition}{definition}{definitions}
\Crefname{definition}{Definition}{Definitions}
\crefname{example}{example}{examples}
\Crefname{example}{Example}{Examples}
\crefname{remark}{remark}{remarks}
\Crefname{remark}{Remark}{Remarks}

\defbibheading{bibliography}[\refname]{%
  \section*{#1}%
}

\title[Special Lagrangian Cones in Deep Learning]{Special Lagrangian Cones in Deep Learning}

\author{Tejas Kotwal}
\address{Principia, London, EC2A 4XE, United Kingdom}
\email{tejas@principia-ai.org}

\author{Govind Menon}
\address{Division of Applied Mathematics, Brown University, Providence, RI 02912, USA}\thanks{Supported by the NSF grant DMS 2407055 and the Erik Ellentuck Fellowship at the Institute for Advanced Study, Princeton. }
\email{govind\_menon@brown.edu}


\begin{document}

%\linenumbers

\begin{abstract}
We introduce a matrix generalization of the cone of Harvey and Lawson
and prove that it is an exact special Lagrangian manifold. We further
show that it belongs to a family of exact special Lagrangian manifolds
that foliate the balanced manifold arising in deep learning.
\end{abstract}

\maketitle

\begin{center}
\emph{For Camillo De Lellis on the occasion of his 50th birthday.}
\end{center}

\section{Overview}
This paper presents examples of minimal varieties arising in deep learning. Our main insight is that the fundamental concept of {\em balancedness\/} in deep learning has close ties to geometric analysis. In particular, balanced varieties in the Deep Linear Network (DLN) provide a natural matrix-valued generalization of the cone of Harvey and Lawson.

%We introduce a matrix generalization of the cone of Harvey and Lawson
%and show that it is an exact special Lagrangian manifold. The
%construction is motivated by the notion of \emph{balancedness}, which plays a
%central role in the geometry of deep learning.

\subsection{Main results}
\label{subsec:main-results}

Fix integers \(N\geq 2\) and \(d\geq 1\).
We write \(\Mat_d(\C)\) for the space of \(d\times d\) complex matrices,
\(\GL_d(\C)\) and \(\U_d\) for the general linear and unitary groups,
and \(\Herm_d^+\) for the cone of positive-definite Hermitian matrices.

Writing $\mathbf W=(W_1,\ldots,W_N)$, we define the smooth submanifold
\begin{equation}
  L_+
  :=
  \left\{
    \mathbf W\in\GL_d(\C)^N
    \ \middle|\
    \begin{aligned}
      W_kW_k^*
      &=
      W_{k+1}^*W_{k+1},
      &&1\leq k<N,\\
      W_N\cdots W_1
      &\in
      \Herm_d^+
    \end{aligned}
  \right\}.
  \label{eq:matrix-harvey-lawson-locus}
\end{equation}

We equip \(\Mat_d(\C)^N\) with its standard Calabi--Yau structure.
Let \(\omega\) denote its symplectic form and \(\alpha\) the standard primitive, so that
\(
  \dd\alpha=\omega.
\)

\begin{theorem}
  \label{thm:identity-leaf}
  The submanifold \(L_+\subset\Mat_d(\C)^N\) is special Lagrangian and
  exact, with \(\alpha|_{L_+}=0\). With a suitable orientation, its
  phase is
  \begin{equation}
    \theta
    =
    \frac{\pi}{2}
    \left(
      d(N-1)+N\binom{d}{2}
    \right)
    \pmod{2\pi}.
    \label{eq:phase-main}
  \end{equation}
Hence \(L_+\) is calibrated and therefore minimal and locally
volume-minimizing.
\end{theorem}
The cone of Harvey and Lawson in \(\C^N\) is
\begin{equation}
  C_{\mathrm{HL}}
  :=
  \left\{
    (z_1,\ldots,z_N)\in\C^N
    \ \middle|\
    |z_1|=\cdots=|z_N|,
    \quad
    z_N\cdots z_1\in\R_{\geq 0}
  \right\}.
  \label{eq:harvey-lawson-scalar}
\end{equation}
Harvey and Lawson proved that \(C_{\mathrm{HL}}\) is special Lagrangian
\cite[Theorem 3.1]{HarveyLawson1982}. When \(d=1\),
\begin{equation}
  L_+
  =
  C_{\mathrm{HL}}\setminus\{\mathbf 0\}.
  \label{eq:scalar-harvey-lawson-identification}
\end{equation}
Thus adjoining the vertex to \(L_+\) recovers \(C_{\mathrm{HL}}\).

The defining conditions for \(L_+\) are the matrix analogues of the
two conditions in \eqref{eq:harvey-lawson-scalar}. The relation
\(|z_k|=|z_{k+1}|\) can be written as
\(z_k\overline{z_k}=\overline{z_{k+1}}z_{k+1}\), suggesting
\begin{equation}
  W_kW_k^*
  =
  W_{k+1}^*W_{k+1},
  \qquad
  1\leq k<N.
  \label{eq:balancedness-equations}
\end{equation}
For \(W_1,\ldots,W_N\in\GL_d(\C)\), the corresponding condition on the
\emph{ordered} product is
\(
  W_N\cdots W_1\in\Herm_d^+
\).

For \(d>1\), the diagonal locus of \(L_+\) is a product of \(d\) copies
of the punctured Harvey--Lawson cone. Indeed, under the identification
\(
  \Diag_d(\C)^N\cong(\C^N)^d,
\)
we have
\begin{equation}
  L_+\cap\Diag_d(\C)^N
  \cong\bigl(C_{\mathrm{HL}}\setminus\{ \mathbf 0\}\bigr)^d.
  \label{eq:diagonal-locus}
\end{equation}

In \Cref{sec:leaf-closures}, we show that the closure of \(L_+\) in
\(\Mat_d(\C)^N\) is the matrix generalization of \(C_{\mathrm{HL}}\)
obtained by allowing the matrices \(W_k\) to be singular and the product
\(W_N\cdots W_1\) to be positive \emph{semidefinite}.

We now consider the real vector space
\(
  \Mat_d(\R)^N,
\)
equipped with its standard Euclidean metric.
We write \(O_d\) for the orthogonal group and \(\Sym_d^+\) for the cone
of positive-definite symmetric \(d\times d\) matrices. The real balanced
manifold is
\begin{equation}
  \mathcal M_{\R}
  :=
  \left\{
    \mathbf W\in\GL_d(\R)^N
    \ \middle|\
    W_kW_k^{T}
    =
    W_{k+1}^{T}W_{k+1},
    \quad
    1\leq k<N
  \right\}.
  \label{eq:real-balanced-manifold}
\end{equation}
Consider the submanifold
\begin{equation}
  L_{+,\R}
  :=
  \left\{
    \mathbf W\in\mathcal M_{\R}
    \ \middle|\
    W_N\cdots W_1\in\Sym_d^+
  \right\}.
  \label{eq:real-positive-locus}
\end{equation}

Unlike in the complex case, the minimality of \(L_{+,\R}\) does not
follow from a calibration and hence requires
a separate argument.

\begin{theorem}
  \label{thm:real-identity-fiber}
  The submanifold \(L_{+,\R}\subset \Mat_d(\R)^N\) is minimal with respect
  to the Euclidean metric.
\end{theorem}

\subsection{Balancedness: a bridge between geometry and optimization}
\label{subsec:balancedness}

The quadratic equations \eqref{eq:balancedness-equations} are known
as the equations for \emph{balancedness} in deep learning. They arise
naturally in the Deep Linear Network (DLN), a basic model for
the geometry of overparametrization (see~\cite{menon_2025} for a review).
A DLN of depth \(N\) and width \(d\) over \(\C\) is described by the
multiplication map
\begin{equation}
  \phi:\Mat_d(\C)^N\longrightarrow\Mat_d(\C),
  \qquad
  \phi(\mathbf W)
  =
  W_N\cdots W_1.
  \label{eq:multiplication-map}
\end{equation}
For any smooth loss \(E:\Mat_d(\C)\to\R\), the Euclidean gradient flow
of \(E\circ\phi\) satisfies the conservation laws \cite{arora_cohen_hazan_2018,MenonYu2025}
\[
  \frac{\dd}{\dd t}
  \left(
    W_kW_k^*-W_{k+1}^*W_{k+1}
  \right)
  =
  0,
  \qquad
  1\leq k<N.
\]

These conserved quantities are the components of a \emph{moment map},
which vanishes precisely when the network is balanced.
The Kempf--Ness theorem characterizes balancedness in terms of
norm minimization \cite{lindsey_menon_2026}.

We call
\begin{equation}
  \mathcal M_{\mathbf 0}
  :=
  \left\{
    \mathbf W\in\Mat_d(\C)^N
    \ \middle|\
    W_kW_k^*
    =
    W_{k+1}^*W_{k+1},
    \quad
    1\leq k<N
  \right\}
  \label{eq:balanced-variety}
\end{equation}
the \emph{balanced variety}. Its intersection with
\(\GL_d(\C)^N\),
\begin{equation}
  \mathcal M
  :=
  \mathcal M_{\mathbf 0}\cap\GL_d(\C)^N,
  \label{eq:balanced-manifold}
\end{equation}
is called the \emph{balanced manifold}.

The factorization of \(\phi(\mathbf W)\) is not unique. The group \(\GL_d(\C)^{N-1}\) acts on
\(\Mat_d(\C)^N\) by changes of basis between adjacent matrices.
For \(g=(g_1,\ldots,g_{N-1})\in\GL_d(\C)^{N-1}\), the action is
\begin{equation}
  (g\cdot\mathbf W)_k
  =
  g_kW_kg_{k-1}^{-1},
  \qquad
  1\leq k\leq N,
  \label{eq:complex-action}
\end{equation}
where \(g_0=g_N=\Id_d\).
These changes of basis cancel in the product so that
\[
  \phi(g\cdot\mathbf W)=\phi(\mathbf W).
\]
The subgroup \(\U_d^{N-1}\subset\GL_d(\C)^{N-1}\) consists of
unitary changes of basis. Its action is Hamiltonian with respect
to the standard symplectic form on \(\Mat_d(\C)^N\)
\cite{lindsey_menon_2026,Kotwal2026}.
The corresponding moment map
\(
  \mu:\Mat_d(\C)^N\longrightarrow\Herm_d^{N-1}
\)
has components
\begin{equation}
  \mu_k(\mathbf W)
  =
  W_kW_k^*-W_{k+1}^*W_{k+1},
  \qquad
  1\leq k<N,
  \label{eq:moment-map}
\end{equation}
where \(\Herm_d\) denotes the space of Hermitian \(d\times d\)
matrices. Thus
\begin{equation}
  \mathcal M_{\mathbf 0}
  =
  \mu^{-1}(\mathbf 0).
  \label{eq:balanced-variety-moment-map}
\end{equation}
The conservation laws state that every level set of \(\mu\) is
invariant under the Euclidean gradient flow of \(E\circ\phi\).

Balancedness is selected by a variational principle. For a fixed \(X\in\GL_d(\C)\), let \(\mathcal F_X=\phi^{-1}(X)\) be the set of all factorizations of \(X\). By the Kempf--Ness theorem,
the factorizations of \(X\) that have minimum Euclidean norm
are precisely those in \(\mathcal M \cap \mathcal F_X\)
\cite[Theorem 1]{lindsey_menon_2026}.

In \Cref{thm:foliation}, we show that the balanced manifold
\(\mathcal M\) is foliated by exact special Lagrangian submanifolds obtained by fixing the \emph{unitary polar factor} of \(X\).
The leaf for which this unitary factor is the identity is the
submanifold introduced in \eqref{eq:matrix-harvey-lawson-locus},
\begin{equation}
  L_+
  =
  \mathcal M\cap\phi^{-1}(\Herm_d^+).
  \label{eq:positive-leaf-balanced-description}
\end{equation}
Thus the balanced manifold $\mathcal M$ carries two distinct notions of minimality.
Its intersection with each fiber of \(\phi\) consists of factorizations
of minimum Euclidean norm, while its special Lagrangian leaves are
calibrated and locally volume-minimizing.

\subsection{Geometric analysis and optimization}
The appearance of fundamental geometric structure in deep learning is, at first, surprising. However, as we discuss below, Calabi-Yau geometry and K\"{a}hler-Einstein metrics are of fundamental importance in the study of interior point algorithms for conic programs. While the nature of the gradient flows in deep learning is very different from that in conic programs, the study of the DLN reveals again the importance of geometric analysis as a mathematical foundation for deep learning. We provide some specific connections between our work and the literature to support this viewpoint.

\begin{remark}[Yang--Mills theory]
  The ADHM construction describes Yang--Mills instantons on
  \(\R^4\) in terms of matrices satisfying quadratic relations
  \cite{atiyah_hitchin_drinfeld_manin}.
  These relations can be written as the vanishing of a moment
  map. Kronheimer and Nakajima extend this idea to instantons
  on ALE spaces, where they fix a value of the moment map and
  identify solutions related by unitary changes of basis
  \cite{KronheimerNakajima1990}. In this sense, the equations~\eqref{eq:moment-map}  for balancedness in the DLN may be seen as the equivalent of the ADHM construction for the DLN.

  In recent work,  Pfaff and Uhlenbeck have shown that the 
  infinite-depth $(N\to \infty)$ limit of the DLN is a dimensionally reduced problem of Yang--Mills theory on a torus \cite{Pfaff-Uhlenbeck,Uhlenbeck26}.
  A treatment of the corresponding one-dimensional gauge theory,
  in which changes of basis become \emph{based gauge
  transformations} on the interval, can be found in
  \cite[Chapters~8--9]{Kotwal2026}.
  Further developments will appear in forthcoming work of the
  authors.
\end{remark}

\begin{remark}[The manifold hypothesis and Occam's razor]
The {\em manifold hypothesis\/} in deep learning is the convenient ansatz that high-dimensional data clouds are effectively described by lower-dimensional manifolds. The justification of this ansatz is a difficult problem and has partly inspired a probabilistic approach to the Nash embedding theorems in the work of one of the authors (GM)~\cite{Inauen-Menon}. 

The minimum principle furnished by the Kempf--Ness theorem
  may be viewed as a form of Occam's razor. Among the many
  factorizations of a fixed end-to-end matrix \(X\), the
  balanced factorizations are selected as the simplest
  parametric representations of \(X\), in the sense of
  minimum Euclidean norm. This selection may be understood
  as model reduction in the DLN
  \cite[\S1.2]{lindsey_menon_2026}.  
\end{remark}

\begin{remark}[Conic programs and self-concordant barriers]
  Karmarkar first noted the importance of the Riemannian geometry underlying interior-point methods for linear programming
  \cite{karmarkar_1990}. Bayer and Lagarias developed this
  viewpoint through their analysis of affine and projective
  scaling trajectories \cite{bayer_lagarias_1989a}.
  
  A connection between conic optimization and minimal
  Lagrangian geometry appears in the study of proper convex
  cones $K \subset \mathbb{R}^n$. Cheng and Yau proved that the interior of each such cone carries a distinguished convex solution of the Monge–Amp\`ere equation whose level sets form a foliation by complete hyperbolic affine spheres \cite{cheng1982real}. Hildebrand and Fox independently showed that this solution is a self-concordant barrier for conic optimization, known as the \emph{canonical barrier} \cite{hildebrand, fox2015schwarz}. The second author and Yu use this structure to interpret the Monge–Amp\`ere equation as arising from equilibrium between intrinsic and extrinsic Brownian motion within the cone $K$ equipped with its Cheng-Yau metric~\cite{menon-yu-conic}. 

  The affine spheres associated with the canonical barrier
  admit a realization as minimal Lagrangian submanifolds of
  a para-K\"ahler space form, as shown by Hildebrand
  \cite[Corollary~4.1]{hildebrand2011crossratio}.
  Fox's corresponding construction in flat para-K\"ahler
  space produces a Lagrangian cone with vanishing mean
  curvature \cite[Theorem~4.3]{fox2015schwarz}. 
\end{remark}

\begin{remark}
  Another matrix generalization of the Harvey--Lawson cone,
  studied by Castro and Urbano \cite{urbano_castro_2004}, is
  \begin{equation}
    C_{\SU_d}
    :=
    \{rU:r\geq0,\ U\in\SU_d\}
    \subset\Mat_d(\C).
    \label{eq:SU-cone}
  \end{equation}
  This is the cone over \(\SU_d\) in a single copy of
  \(\Mat_d(\C)\). At its nonzero points, the unitary polar factor
  varies over \(\SU_d\), while the positive polar factor is
  confined to the one-dimensional family
  \(\{r\Id_d:r>0\}\).

  By contrast, \(L_+\subset\Mat_d(\C)^N\) fixes the unitary polar
  factor of the product to \(\Id_d\), while its positive polar
  factor varies over all of \(\Herm_d^+\). It also retains the
  ordered factorization into \(N\) matrices.
  For other constructions of special Lagrangian submanifolds and
  cones with symmetry, see \cite{haskins_2004,Joyce_2002}.
\end{remark}

\subsection{Organization of the paper}

In \Cref{sec:polar-decomposition}, we prove a common polar
factorization and use it to obtain global coordinates on the balanced
manifold \(\mathcal M\). 
In \Cref{sec:proofs}, we establish that \(L_+\) is Lagrangian and
exact, compute its phase, and extend these results to the full family
\(L_Q\). The section concludes with the corresponding minimality
results in the real setting.
In \Cref{sec:leaf-closures}, we study the closures
\(\mathcal C_Q=\overline{L_Q}\) and their relation to the balanced
variety \(\mathcal M_{\mathbf0}\).
\section{Polar decomposition}
\label{sec:polar-decomposition}

The equations for balancedness in \eqref{eq:balancedness-equations} imply that
the matrices \(W_1,\ldots,W_N\) have the same singular values. The following
proposition strengthens this observation by identifying a single
positive-definite Hermitian matrix common to all the matrices. We also obtain
a global parametrization of \(\mathcal M\).

\begin{proposition}
  \label[proposition]{prop:balanced-coordinates}
  For every $\mathbf{W}\in\mathcal M$, there exist unitary matrices
  $V_0,\ldots,V_N\in\U_d$ and a matrix $P\in\Herm_d^+$ such that
  \begin{equation}
    W_k=V_kPV_{k-1}^*,
    \qquad 1\leq k\leq N.
    \label{eq:common-factorization}
  \end{equation}
  After imposing $V_0=\Id_d$, the factorization is unique. Consequently, the map
  \begin{equation}
    \begin{aligned}
      \Psi:\Herm_d^+\times\U_d^N &\longrightarrow\mathcal M,\\
      (P,V_1,\ldots,V_N) &\longmapsto
      \bigl(V_1P,\,V_2PV_1^*,\ldots,V_NPV_{N-1}^*\bigr)
    \end{aligned}
    \label{eq:Psi}
  \end{equation}
  is a diffeomorphism.
\end{proposition}

\begin{proof}
  Let
  \(
    \mathbf W=(W_1,\ldots,W_N)\in\mathcal M.
  \)
  For each \(1\leq k\leq N\), write the polar decomposition
  \begin{equation}
    W_k=A_kH_k,
    \qquad
    A_k\in\U_d,
    \qquad
    H_k:=(W_k^*W_k)^{1/2}\in\Herm_d^+.
    \label{eq:layerwise-polar-decomposition}
  \end{equation}
  The equations for balancedness imply
  \[
    H_{k+1}^2
    =
    W_{k+1}^*W_{k+1}
    =
    W_kW_k^*
    =
    A_kH_k^2A_k^*,
    \qquad
    1\leq k<N.
  \]
  Both sides are positive definite, so uniqueness of the
  positive-definite square root gives
  \begin{equation}
    H_{k+1}=A_kH_kA_k^*,
    \qquad
    1\leq k<N.
    \label{eq:successive-positive-factors}
  \end{equation}

  Choose \(V_0\in\U_d\), and set
  \[
    P:=V_0^*H_1V_0,
    \qquad
    V_k:=A_kV_{k-1},
    \quad
    1\leq k\leq N.
  \]
  We claim that
  \begin{equation}
    H_k=V_{k-1}PV_{k-1}^*,
    \qquad
    1\leq k\leq N.
    \label{eq:positive-factors-from-common-P}
  \end{equation}
  This is immediate for \(k=1\). If it holds for \(k\), then
  \eqref{eq:successive-positive-factors} gives
  \[
    H_{k+1}
    =
    A_kV_{k-1}PV_{k-1}^*A_k^*
    =
    V_kPV_k^*.
  \]
  Thus \eqref{eq:positive-factors-from-common-P} follows by induction.
  Substituting it into the polar decomposition of \(W_k\), we obtain
  \[
    W_k
    =
    A_kH_k
    =
    A_kV_{k-1}PV_{k-1}^*
    =
    V_kPV_{k-1}^*,
    \qquad
    1\leq k\leq N.
  \]
  This proves the existence of the common factorization
  \eqref{eq:common-factorization}.

  We next prove uniqueness when \(V_0=\Id_d\).
  The first matrix satisfies \(W_1=V_1P\), and hence
  \(
    W_1^*W_1=P^2.
  \)
  Since \(P\) is positive definite, it is uniquely determined by
  \(P=(W_1^*W_1)^{1/2}\).
  Once \(P\) and \(V_{k-1}\) are known, the \(k\)-th matrix uniquely
  determines
  \begin{equation}
    V_k=W_kV_{k-1}P^{-1},
    \qquad
    1\leq k\leq N.
    \label{eq:recover-unitary-factors}
  \end{equation}
  Thus the factorization is unique.

  Conversely, let \(P\in\Herm_d^+\) and
  \(V_1,\ldots,V_N\in\U_d\), set \(V_0=\Id_d\), and define
  \(
    W_k=V_kPV_{k-1}^*.
  \)
  Then
  \[
    W_kW_k^*
    =
    V_kP^2V_k^*
    =
    W_{k+1}^*W_{k+1},
    \qquad
    1\leq k<N.
  \]
  Hence \(\mathbf W\in\mathcal M\), so the map \(\Psi\) in
  \eqref{eq:Psi} is well defined. The existence and uniqueness proved
  above show that \(\Psi\) is bijective.

  Finally, \(\Psi\) is smooth, while
  \eqref{eq:recover-unitary-factors} gives its inverse recursively.
  The positive-definite square-root map is smooth on \(\Herm_d^+\),
  and matrix multiplication and inversion are smooth. Therefore
  \(\Psi^{-1}\) is smooth, and \(\Psi\) is a diffeomorphism.
\end{proof}

In particular,
\begin{equation}
  \mathcal M
  \cong
  \Herm_d^+\times\U_d^N,
  \qquad
  \dim_{\R}\mathcal M=(N+1)d^2.
  \label{eq:normalized-balanced-coordinates}
\end{equation}

With \(V_0=\Id_d\), the product telescopes to
\begin{equation}
  W_N\cdots W_1
  =
  V_NP^N.
  \label{eq:ordered-product-polar-decomposition}
\end{equation}
Since \(V_N\) is unitary and \(P^N\) is positive-definite Hermitian,
\eqref{eq:ordered-product-polar-decomposition} is the polar decomposition
of the product. By uniqueness of the polar decomposition,
\begin{equation}
  W_N\cdots W_1\in\Herm_d^+
  \quad\Longleftrightarrow\quad
  V_N=\Id_d.
  \label{eq:positive-product-terminal-unitary}
\end{equation}
Consequently, the submanifold \(L_+\) introduced in
\eqref{eq:matrix-harvey-lawson-locus} is given by
\begin{equation}
  L_+
  =
  \Psi\bigl(
    \Herm_d^+\times\U_d^{N-1}\times\{\Id_d\}
  \bigr)
  \cong
  \Herm_d^+\times\U_d^{N-1},
  \qquad
  \dim_{\R}L_+=Nd^2.
  \label{eq:normalized-identity-leaf-coordinates}
\end{equation}

For \(P\in\Herm_d^+\), define
\begin{equation}
  \Delta(P):=(P,\ldots,P)\in L_+.
  \label{eq:diagonal-point}
\end{equation}
Every point of \(L_+\) can be written uniquely as
\(
  \mathbf W
  =
  \Psi(P,V_1,\ldots,V_{N-1},\Id_d).
\)
For
\(
  g=(V_1^*,\ldots,V_{N-1}^*)\in\U_d^{N-1},
\)
the action \eqref{eq:complex-action} gives
\(
  g\cdot\mathbf W=\Delta(P).
\)
Thus every point of \(L_+\) can be moved by the
\(\U_d^{N-1}\)-action to the distinguished point \(\Delta(P)\).

More generally, the unitary polar factor of the product defines
a smooth map
\begin{equation}
  \begin{aligned}
    q:\mathcal M
    &\longrightarrow
    \U_d,\\
    \mathbf W
    & \longmapsto
    \bigl(W_N\cdots W_1\bigr)
    \left(
      \bigl(W_N\cdots W_1\bigr)^*
      \bigl(W_N\cdots W_1\bigr)
    \right)^{-1/2}.
  \end{aligned}
  \label{eq:polar-factor-map}
\end{equation}
In the coordinates of \Cref{prop:balanced-coordinates},
\begin{equation}
  (q\circ\Psi)(P,V_1,\ldots,V_N)
  =
  V_N.
  \label{eq:q-in-balanced-coordinates}
\end{equation}
When \(d=1\), this reduces to
\begin{equation}
  q(z_1,\ldots,z_N)
  =
  \frac{z_N\cdots z_1}{|z_N\cdots z_1|},
  \label{eq:scalar-phase-map}
\end{equation}
the usual phase of the product $z_N\cdots z_1$.

For each \(Q\in\U_d\), define
\begin{equation}
  \begin{aligned}
    L_Q
    :=
    q^{-1}(Q)
    =
    \left\{
      \mathbf W\in\mathcal M
      \ \middle|\
      W_N\cdots W_1\in Q \Herm_d^+
    \right\}.
  \end{aligned}
  \label{eq:polar-leaves}
\end{equation}
By construction,
\(
  L_{\Id_d}=L_+.
\)

\begin{theorem}
  \label[theorem]{thm:foliation}
  The map
  \(
    q:\mathcal M\longrightarrow\U_d
  \)
  is a smooth globally trivial fiber bundle whose fiber over
  \(Q\in\U_d\) is \(L_Q\). In particular,
  \begin{equation}
    \mathcal M
    =
    \bigsqcup_{Q\in\U_d}L_Q,
    \qquad
    \mathcal M\cong L_+\times\U_d,
    \label{eq:foliation}
  \end{equation}
  and every fiber \(L_Q\) is an exact special Lagrangian submanifold of
  \(\Mat_d(\C)^N\). With a suitable orientation, its phase is
  \[
    \theta_Q
    =
    \theta+d\,\arg\det Q
    \pmod{2\pi}.
  \]
\end{theorem}

The proof of \Cref{thm:foliation} is given in
\Cref{subsec:proof-special-lagrangian-fibers}.

When \(d=1\), the fibers \(L_Q\) are phase rotations of the
punctured Harvey--Lawson cone. For \(d>1\), the unitary matrix \(Q\)
determines the Lagrangian submanifold \(L_Q\), while its phase
depends only on \(\det Q\).
\[
  e^{i\theta_Q}=e^{i\theta}(\det Q)^d.
\]
In particular, \(L_Q\) has the same phase as \(L_+\) for every
\(Q\in\SU_d\).

For the real setting, the same argument as in
\Cref{prop:balanced-coordinates}, with unitary matrices replaced
by orthogonal matrices, gives a factorization
\begin{equation}
  W_k=V_kPV_{k-1}^{T},
  \qquad
  1\leq k\leq N,
  \label{eq:real-common-factorization}
\end{equation}
where \(P\in\Sym_d^+\) and \(V_0,\ldots,V_N\in O_d\).
After imposing \(V_0=\Id_d\), this factorization is unique.

Consequently, the map
\begin{equation}
  \begin{aligned}
    \Psi_{\R}:\Sym_d^+\times O_d^N
    &\longrightarrow\mathcal M_{\R},\\
    (P,V_1,\ldots,V_N)
    &\longmapsto
    \bigl(
      V_1P,\,
      V_2PV_1^T,\ldots,
      V_NPV_{N-1}^T
    \bigr)
  \end{aligned}
  \label{eq:real-Psi}
\end{equation}
is a diffeomorphism. In these coordinates, the product of the
components of \(\Psi_{\R}(P,V_1,\ldots,V_N)\) is
\(
  V_NP^N.
\)

Consequently,
\[
  L_{+,\R}
  =
  \Psi_{\R}\bigl(
    \Sym_d^+\times O_d^{N-1}\times\{\Id_d\}
  \bigr).
\]

The full family is again indexed by the orthogonal polar factor
of the product. For \(\mathbf W\in\mathcal M_{\R}\), define
the smooth map
\begin{equation}
  q_{\R}:\mathcal M_{\R}\longrightarrow O_d,
  \qquad
  \mathbf W
  \longmapsto
  (W_N\cdots W_1)
  \bigl(
    (W_N\cdots W_1)^{T}
    (W_N\cdots W_1)
  \bigr)^{-1/2}.
  \label{eq:real-polar-factor-map}
\end{equation}
For \(Q\in O_d\), define
\begin{equation}
  L_{Q,\R}
  :=
  q_{\R}^{-1}(Q)
  =
  \left\{
    \mathbf W\in\mathcal M_{\R}
    \ \middle|\
    W_N\cdots W_1\in Q\Sym_d^+
  \right\}.
  \label{eq:real-phase-fibers}
\end{equation}
By construction, \(L_{\Id_d,\R}=L_{+,\R}\).
In these coordinates, \(q_{\R}\) is projection onto the final
\(O_d\)-factor.

\begin{corollary}
  \label[corollary]{cor:real-fibration}
  The map
  \(
    q_{\R}:\mathcal M_{\R}\longrightarrow O_d
  \)
  is a smooth globally trivial fiber bundle whose fiber over
  \(Q\in O_d\) is \(L_{Q,\R}\). In particular,
  \[
    \mathcal M_{\R}
    =
    \bigsqcup_{Q\in O_d}L_{Q,\R},
    \qquad
    \mathcal M_{\R}
    \cong
    L_{+,\R}\times O_d.
  \]
  Each \(L_{Q,\R}\) has exactly \(2^{N-1}\) connected components,
  all of which are minimal in \(\Mat_d(\R)^N\).
  These components form a foliation of \(\mathcal M_{\R}\).
\end{corollary}

The proof of \Cref{cor:real-fibration} is given in
\Cref{subsec:proof-real-case}.
\section{Proofs}
\label{sec:proofs}

\subsection{Special Lagrangian submanifolds}
\label{subsec:proof-special-lagrangian-fibers}

We follow the conventions of Harvey and Lawson
\cite{HarveyLawson1982}.
We equip \(\Mat_d(\C)^N\cong\C^{Nd^2}\) with its standard
Calabi--Yau structure. Its symplectic form is
\begin{equation}
  \omega(\mathbf Z,\mathbf Y)
  =
  \operatorname{Im}
  \sum_{k=1}^N\Tr(Z_k^*Y_k),
  \qquad
  \mathbf Z,\mathbf Y\in\Mat_d(\C)^N.
  \label{eq:symplectic-form}
\end{equation}
After fixing an order on the matrix entries, its holomorphic volume
form is
\begin{equation}
  \Omega
  =
  \bigwedge_{k=1}^N\ \bigwedge_{a,b=1}^d
  \dd(W_k)_{ab}.
  \label{eq:holomorphic-volume-form}
\end{equation}
An oriented real \(Nd^2\)-dimensional submanifold
\(L\subset\Mat_d(\C)^N\) is special Lagrangian of phase \(\theta\) if
\[
  \omega|_L=0,
  \qquad
  \operatorname{Im}\bigl(e^{-i\theta}\Omega\bigr)|_L=0,
\]
where the orientation is chosen so that
\(\operatorname{Re}(e^{-i\theta}\Omega)|_L\) is positive.

We use the primitive
\begin{equation}
  \alpha
  =
  \frac12
  \operatorname{Im}
  \sum_{k=1}^N\Tr(W_k^*\,\dd W_k),
  \qquad
  \dd\alpha=\omega.
  \label{eq:standard-primitive}
\end{equation}
A Lagrangian submanifold \(L\subset\Mat_d(\C)^N\) is exact if
\(\alpha|_L\) is exact.

We first prove \Cref{thm:identity-leaf}. We begin by describing the
tangent space of \(L_+\) at \(\Delta(P)\).

\begin{lemma}
  \label[lemma]{lem:identity-leaf-tangent-space}
  Let \(P\in\Herm_d^+\). Every tangent vector
  \(\mathbf Z=(Z_1,\ldots,Z_N)\in T_{\Delta(P)}L_+\) can be written
  uniquely as
  \begin{equation}
    Z_k
    =
    A+X_kP-PX_{k-1},
    \qquad
    1\leq k\leq N,
    \label{eq:identity-leaf-tangent-space}
  \end{equation}
  where
  \[
    A\in\Herm_d,
    \qquad
    X_1,\ldots,X_{N-1}\in\mathfrak u_d,
    \qquad
    X_0=X_N=0.
  \]
  Consequently,
  \begin{equation}
    T_{\Delta(P)}L_+
    =
    \left\{
      (A,\ldots,A):A\in\Herm_d
    \right\}
    \oplus
    T_{\Delta(P)}
    \bigl(\U_d^{N-1}\cdot\Delta(P)\bigr).
    \label{eq:identity-leaf-tangent-decomposition}
  \end{equation}
\end{lemma}

\begin{proof}
  Under the coordinates in
  \eqref{eq:normalized-identity-leaf-coordinates}, the point
  \(\Delta(P)\) corresponds to
  \(
    (P,\Id_d,\ldots,\Id_d).
  \)
  Let \(P(t)\in\Herm_d^+\) and \(V_k(t)\in\U_d\) be smooth curves such
  that
  \[
    P(0)=P,
    \qquad
    V_k(0)=\Id_d,
    \qquad
    1\leq k<N.
  \]
  Write
  \[
    A=\dot P(0)\in\Herm_d,
    \qquad
    X_k=\dot V_k(0)\in\mathfrak u_d,
    \qquad
    X_0=X_N=0.
  \]
  Differentiating
  \[
    W_k(t)=V_k(t)P(t)V_{k-1}(t)^*
  \]
  at \(t=0\) gives
  \[
    \dot W_k(0)
    =
    A+X_kP-PX_{k-1}.
  \]
  Every choice of \(A,X_1,\ldots,X_{N-1}\) arises in this way.
  Uniqueness follows because the parametrization of \(L_+\) in
  \eqref{eq:normalized-identity-leaf-coordinates} is a
  diffeomorphism.

  The terms obtained by setting \(A=0\) are precisely the
  infinitesimal \(\U_d^{N-1}\)-action at \(\Delta(P)\). This gives
  \eqref{eq:identity-leaf-tangent-decomposition}.
\end{proof}

\begin{proof}[Proof of \Cref{thm:identity-leaf}]
  The action of \(\U_d^{N-1}\) defined in \eqref{eq:complex-action}
  preserves \(L_+\), since it preserves both \(\mathcal M\) and the product. It acts unitarily on \(\Mat_d(\C)^N\), and hence
  preserves the symplectic form \(\omega\). Every point of \(L_+\) can
  be moved by this action to a point \(\Delta(P)\). It is therefore
  enough to show that \(T_{\Delta(P)}L_+\) is a Lagrangian subspace of
  \(T_{\Delta(P)}\Mat_d(\C)^N\).

  Let \(\boldsymbol\xi\in\mathfrak u_d^{N-1}\), and let
  \(\boldsymbol\xi^\#\) denote the corresponding infinitesimal vector
  field on \(\Mat_d(\C)^N\). For every
  \(\mathbf Y\in T_{\Delta(P)}\mathcal M\), the moment-map identity
  gives
  \[
    \omega\bigl(
      \boldsymbol\xi^\#(\Delta(P)),
      \mathbf Y
    \bigr)
    =
    \dd\langle\mu,\boldsymbol\xi\rangle(\mathbf Y)
    =
    0,
  \]
  since \(\mu\) vanishes identically on \(\mathcal M\). Thus the orbit
  directions
  \[
    T_{\Delta(P)}
    \bigl(\U_d^{N-1}\cdot\Delta(P)\bigr)
  \]
  pair trivially under \(\omega\) with every tangent vector to
  \(\mathcal M\), and in particular with every tangent vector to
  \(L_+\).

  It remains to pair the first summands in
  \eqref{eq:identity-leaf-tangent-decomposition}. For
  \(A,B\in\Herm_d\),
  \[
    \omega\bigl(
      (A,\ldots,A),
      (B,\ldots,B)
    \bigr)
    =
    N\operatorname{Im}\Tr(AB)
    =
    0,
  \]
  because \(\Tr(AB)\) is real. It follows that
  \(\omega|_{T_{\Delta(P)}L_+}=0\). The
  \(\U_d^{N-1}\)-invariance then gives
  \(\omega|_{L_+}=0\). Since
  \[
    \dim_{\R}L_+
    =
    Nd^2
    =
    \frac12\dim_{\R}\Mat_d(\C)^N
  \]
  by \eqref{eq:normalized-identity-leaf-coordinates}, the submanifold
  \(L_+\) is Lagrangian.

  We next prove exactness. The set \(L_+\) is invariant under positive
  scaling, so the radial vector field
  \(
    R(\mathbf W)=\mathbf W
  \)
  is tangent to \(L_+\). Since
  \(
    \alpha=\frac12\iota_R\omega,
  \)
  for every \(\mathbf Z\in T_{\mathbf W}L_+\) we have
  \[
    \alpha_{\mathbf W}(\mathbf Z)
    =
    \frac12
    \omega_{\mathbf W}
    \bigl(R(\mathbf W),\mathbf Z\bigr)
    =
    0.
  \]
  Hence
  \(
    \alpha|_{L_+}=0.
  \)

  It remains to compute the phase. We first record that the
  \(\U_d^{N-1}\)-action preserves the holomorphic volume form
  \(\Omega\). Indeed, for
  \(
    g=(g_1,\ldots,g_{N-1})\in\U_d^{N-1},
  \)
  each \(g_j\) acts by left multiplication on the \(j\)-th matrix and
  by right multiplication by \(g_j^{-1}\) on the \((j+1)\)-st matrix.
  Its two contributions to the complex determinant on
  \(\Mat_d(\C)^N\) are
  \(
    (\det g_j)^d
    \; \text{and} \;
    (\det g_j^{-1})^d.
  \)
  They cancel, and hence the complex determinant of the full action
  is one. Therefore
  \[
    g^*\Omega=\Omega,
    \qquad
    g\in\U_d^{N-1}.
  \]
  It is consequently enough to calculate the phase at \(\Delta(P)\).

  Choose \(U\in\U_d\) such that
  \[
    P=UDU^*,
    \qquad
    D=\operatorname{diag}(p_1,\ldots,p_d),
    \qquad
    p_a>0.
  \]
  Simultaneous unitary conjugation
  \[
    \Gamma_U(W_1,\ldots,W_N)
    :=
    (U^*W_1U,\ldots,U^*W_NU)
  \]
  preserves \(L_+\) and maps \(\Delta(P)\) to \(\Delta(D)\). On each
  copy of \(\Mat_d(\C)\), its complex determinant is
  \(
    (\det U^*)^d(\det U)^d=1.
  \)
  Thus \(\Gamma_U^*\Omega=\Omega\), and it is enough to compute at
  \(\Delta(D)\).

  Let
  \[
    \mathcal T_D:
    \Herm_d\oplus\mathfrak u_d^{N-1}
    \longrightarrow
    T_{\Delta(D)}L_+
  \]
  denote the isomorphism in
  \eqref{eq:identity-leaf-tangent-space}. After permuting the factors
  in \(\Omega\), which changes the result only by a sign, the
  pullback \(\mathcal T_D^*\Omega\) separates into blocks indexed by
  the diagonal entries \(a\) and the unordered pairs \(a<b\).

  Fix a diagonal entry \(a\). Write
  \[
    A_{aa}=r,
    \qquad
    (X_j)_{aa}=ix_j,
    \qquad
    r,x_j\in\R,
    \qquad
    x_0=x_N=0.
  \]
  Then
  \begin{equation}
    (Z_k)_{aa}
    =
    r+ip_a(x_k-x_{k-1}),
    \qquad
    1\leq k\leq N.
    \label{eq:diagonal-tangent-block}
  \end{equation}
  Relative to the real variables
  \((r,x_1,\ldots,x_{N-1})\), the corresponding complex determinant is
  \[
    \det
    \begin{pmatrix}
      1 & ip_a & 0 & \cdots & 0\\
      1 & -ip_a & ip_a & \ddots & \vdots\\
      \vdots & \ddots & \ddots & \ddots & 0\\
      1 & 0 & \cdots & -ip_a & ip_a\\
      1 & 0 & \cdots & 0 & -ip_a
    \end{pmatrix}
    =
    (-1)^{N-1}N(ip_a)^{N-1}.
  \]
  Thus each diagonal entry contributes the phase
  \(
    i^{N-1}.
  \)

  Now fix \(a<b\). Write
  \[
    A_{ab}=u,
    \qquad
    (X_j)_{ab}=\xi_j,
    \qquad
    u,\xi_j\in\C,
    \qquad
    \xi_0=\xi_N=0.
  \]
  Hermitian and skew-Hermitian symmetry give
  \(
    A_{ba}=\overline u,
    \;
    (X_j)_{ba}=-\overline{\xi_j}.
  \)
  Hence
  \begin{align}
    (Z_k)_{ab}
    &=
    u+p_b\xi_k-p_a\xi_{k-1},
    \label{eq:off-diagonal-tangent-block-ab}\\
    (Z_k)_{ba}
    &=
    \overline u-p_a\overline{\xi_k}
    +p_b\overline{\xi_{k-1}}.
    \label{eq:off-diagonal-tangent-block-ba}
  \end{align}
  Define
  \[
    s_N(x,y)
    :=
    \sum_{\ell=0}^{N-1}x^{N-1-\ell}y^\ell
    =
    x^{N-1}+x^{N-2}y+\cdots+y^{N-1}.
  \]
  Since \(p_a,p_b>0\), we have \(s_N(p_a,p_b)>0\). A direct row
  reduction shows that the coefficient matrices of the complex-linear
  maps
  \[
    (u,\xi_1,\ldots,\xi_{N-1})
    \longmapsto
    \bigl(
      (Z_1)_{ab},
      \ldots,
      (Z_N)_{ab}
    \bigr)
  \]
  and
  \[
    (u,\xi_1,\ldots,\xi_{N-1})
    \longmapsto
    \bigl(
      \overline{(Z_1)_{ba}},
      \ldots,
      \overline{(Z_N)_{ba}}
    \bigr)
  \]
  have determinants
  \(
    (-1)^{N-1}s_N(p_a,p_b)
    \)
  and
  \(
    s_N(p_a,p_b),
  \)
  respectively. These determinants are real and nonzero. The
  remaining contribution comes from
  \begin{align*}
    &\dd u\wedge
    \dd\xi_1\wedge\cdots\wedge\dd\xi_{N-1}
    \wedge
    \dd\overline u\wedge
    \dd\overline{\xi_1}\wedge\cdots
    \wedge\dd\overline{\xi_{N-1}}\\
    &\qquad
    =
    c_Ni^N\,
    \dd\operatorname{Re}u\wedge
    \dd\operatorname{Im}u\wedge
    \bigwedge_{j=1}^{N-1}
    \left(
      \dd\operatorname{Re}\xi_j
      \wedge
      \dd\operatorname{Im}\xi_j
    \right)
  \end{align*}
  for some \(c_N\in\R^\times\), where the sign depends only on the
  chosen order of the variables. Thus each unordered pair \(a<b\)
  contributes the phase
  \(
    i^N.
  \)

  There are \(d\) diagonal blocks and \(\binom d2\) off-diagonal
  blocks. Multiplying their contributions gives
  \begin{equation}
    \mathcal T_D^*\Omega
    =
    c(D)\,
    i^{\,d(N-1)+N\binom d2}\,
    \nu_D,
    \qquad
    c(D)\in\R^\times,
    \label{eq:holomorphic-volume-tangent-restriction}
  \end{equation}
  where \(\nu_D\) is a real volume form. The sign of \(c(D)\) depends
  only on the fixed ordering of the matrix entries and tangent
  variables, and not on the positive numbers \(p_1,\ldots,p_d\).

  Choosing the orientation for which the real factor in
  \eqref{eq:holomorphic-volume-tangent-restriction} is positive, we
  obtain
  \[
    \theta
    =
    \frac{\pi}{2}
    \left(
      d(N-1)+N\binom d2
    \right)
    \pmod{2\pi}.
  \]
  Therefore \(L_+\) is special Lagrangian of phase \(\theta\). It is
  calibrated by
  \(\operatorname{Re}(e^{-i\theta}\Omega)\), and hence is minimal and
  locally volume-minimizing
  \cite[\S II.4, Theorem 4.2]{HarveyLawson1982}.
\end{proof}

We now pass from \(L_+\) to \(L_Q\).

\begin{proof}[Proof of \Cref{thm:foliation}]
  In the coordinates of \Cref{prop:balanced-coordinates}, the
  product of the components of \(\Psi(P,V_1,\ldots,V_N)\) is
  \(
    V_NP^N.
  \)
  
  Therefore
  \[
    (q\circ\Psi)(P,V_1,\ldots,V_N)=V_N.
  \]
  It follows that \(q\) is a smooth globally trivial fiber bundle and
  that
  \begin{equation}
    L_Q
    =
    \Psi\bigl(
      \Herm_d^+\times\U_d^{N-1}\times\{Q\}
    \bigr)
    \cong
    \Herm_d^+\times\U_d^{N-1}.
    \label{eq:general-leaf-coordinates}
  \end{equation}
  Under the identification of \(L_+\) in
  \eqref{eq:normalized-identity-leaf-coordinates}, the diffeomorphism
  \(\Psi\) gives the global trivialization in \eqref{eq:foliation}.
  Since \(\Herm_d^+\) and \(\U_d\) are connected, every fiber \(L_Q\)
  is connected. Thus the fibers are the leaves of a foliation of
  \(\mathcal M\).

  For \(Q\in\U_d\), define
  \[
    \tau_Q(W_1,\ldots,W_N)
    =
    (W_1,\ldots,W_{N-1},QW_N).
  \]
  This map preserves balancedness and satisfies
  \[
    (\tau_Q\mathbf W)_N\cdots(\tau_Q\mathbf W)_1
    =
    Q(W_N\cdots W_1),
  \]
  so \(\tau_Q(L_+)=L_Q\). Moreover,
  \[
    \tau_Q^*\omega=\omega,
    \qquad
    \tau_Q^*\alpha=\alpha,
    \qquad
    \tau_Q^*\Omega=(\det Q)^d\Omega.
  \]
  Together with \Cref{thm:identity-leaf}, these identities show that
  \(L_Q\) is special Lagrangian and that
  \(
    \alpha|_{L_Q}=0.
  \)
  In particular, \(L_Q\) is exact. Orient \(L_Q\) by transporting the
  chosen orientation of \(L_+\) under \(\tau_Q\). Its phase is then
  \begin{equation}
    \theta_Q
    =
    \theta+d\,\arg\det Q
    \pmod{2\pi}.
    \label{eq:phase-general-Q}
    \qedhere
  \end{equation}
\end{proof}

\subsection{Minimal submanifolds in the real case}
\label{subsec:proof-real-case}

Recall the coordinates \(\Psi_{\R}\) for the real balanced manifold from
\eqref{eq:real-Psi}.

\begin{proof}[Proof of \Cref{thm:real-identity-fiber}]
  The group \(O_d^{N-1}\) acts on \(\Mat_d(\R)^N\) by the restriction
  of \eqref{eq:complex-action}. This action is orthogonal and preserves
  \(L_{+,\R}\). As in the complex case, every point
  \[
    \mathbf W
    =
    \Psi_{\R}(P,V_1,\ldots,V_{N-1},\Id_d)
  \]
  is mapped by
  \(
    (V_1^T,\ldots,V_{N-1}^T)\in O_d^{N-1}
  \)
  to
  \(
    \Delta(P):=(P,\ldots,P).
  \)
  It is therefore enough to prove that the mean-curvature vector
  vanishes at every \(\Delta(P)\).

  We use two symmetries of \(L_{+,\R}\). First, the cyclic shift
  \[
    c(W_1,\ldots,W_N)
    :=
    (W_2,\ldots,W_N,W_1)
  \]
  is an ambient orthogonal transformation preserving \(L_{+,\R}\).
  Indeed, if
  \[
    \mathbf W
    =
    \Psi_{\R}(P,V_1,\ldots,V_{N-1},\Id_d),
  \]
  set \(V_N=\Id_d\) and
  \[
    \widetilde P:=V_1PV_1^T,
    \qquad
    \widetilde V_k:=V_{k+1}V_1^T,
    \quad
    1\leq k<N.
  \]
  A direct substitution gives
  \[
    c(\mathbf W)
    =
    \Psi_{\R}
    \bigl(
      \widetilde P,
      \widetilde V_1,\ldots,\widetilde V_{N-1},
      \Id_d
    \bigr).
  \]

  Second, the transpose-reversal map
  \[
    \rho(W_1,\ldots,W_N)
    :=
    (W_N^T,\ldots,W_1^T)
  \]
  is an ambient orthogonal transformation preserving \(L_{+,\R}\).
  It preserves the equations for balancedness and satisfies
  \[
    \bigl(\rho(\mathbf W)\bigr)_N\cdots
    \bigl(\rho(\mathbf W)\bigr)_1
    =
    (W_N\cdots W_1)^T.
  \]

  Both \(c\) and \(\rho\) fix \(\Delta(P)\). Write the mean-curvature
  vector at this point as
  \[
    \mathbf H_{\Delta(P)}
    =
    (B_1,\ldots,B_N).
  \]
  Equivariance of mean curvature under \(c\) gives
  \(
    B_1=\cdots=B_N=:B.
  \)
  Equivariance under \(\rho\) then gives
  \(
    B=B^T.
  \)

  For every \(A\in\Sym_d\), the curve
  \(
    t\longmapsto\Delta(P+tA)
  \)
  lies in \(L_{+,\R}\) for sufficiently small \(t\). Hence
  \(
    (A,\ldots,A)
    \in
    T_{\Delta(P)}L_{+,\R}.
  \)
  Since the mean-curvature vector is normal to \(L_{+,\R}\),
  \[
    0
    =
    \left\langle
      \mathbf H_{\Delta(P)},
      (A,\ldots,A)
    \right\rangle
    =
    N\Tr(B^TA)
  \]
  for every \(A\in\Sym_d\). Taking \(A=B\) gives \(B=0\).
  Therefore \(\mathbf H_{\Delta(P)}=0\), and \(L_{+,\R}\) is minimal.
\end{proof}

\begin{proof}[Proof of \Cref{cor:real-fibration}]
  In the coordinates of \eqref{eq:real-Psi}, we have
  \[
    (q_{\R}\circ\Psi_{\R})(P,V_1,\ldots,V_N)=V_N.
  \]
  Hence \(q_{\R}\) is a smooth globally trivial fiber bundle, with
  \[
    L_{Q,\R}
    =
    \Psi_{\R}\bigl(
      \Sym_d^+\times O_d^{N-1}\times\{Q\}
    \bigr)
    \cong
    \Sym_d^+\times O_d^{N-1}.
  \]

  For \(Q\in O_d\), define
  \[
    \tau_Q(W_1,\ldots,W_N)
    :=
    (W_1,\ldots,W_{N-1},QW_N).
  \]
  This is an orthogonal transformation of \(\Mat_d(\R)^N\). It
  preserves balancedness and satisfies
  \[
    (\tau_Q\mathbf W)_N\cdots(\tau_Q\mathbf W)_1
    =
    Q(W_N\cdots W_1),
  \]
  so \(\tau_Q(L_{+,\R})=L_{Q,\R}\). It follows from
  \Cref{thm:real-identity-fiber} that every connected component of
  \(L_{Q,\R}\) is minimal.

  Finally, \(\Sym_d^+\) is connected and \(O_d\) has two connected
  components. Thus \(L_{Q,\R}\) has exactly \(2^{N-1}\) connected
  components. Since \(q_{\R}\) is a submersion, the connected
  components of its fibers form a foliation of \(\mathcal M_{\R}\).
\end{proof}
\section{Closures and low-rank factorizations}
\label{sec:leaf-closures}

In the scalar case, adjoining the origin to \(L_+\) recovers the
Harvey--Lawson cone. This motivates the study of the closure of each fiber
\(L_Q\) in \(\Mat_d(\C)^N\), thereby allowing the common rank of the
matrices to drop. For \(Q\in\U_d\), set
\begin{equation}
  \mathcal C_Q
  :=
  \overline{L_Q},
  \label{eq:closure-definition}
\end{equation}
where the closure is taken in \(\Mat_d(\C)^N\).

For example, when \(N=2\), writing \(A=V_1P\) gives
\begin{equation}
  L_Q
  =
  \left\{
    (A,QA^*):A\in\GL_d(\C)
  \right\},
  \qquad
  \mathcal C_Q
  =
  \left\{
    (A,QA^*):A\in\Mat_d(\C)
  \right\}.
  \label{eq:depth-two-closure}
\end{equation}
In particular, \(\mathcal C_Q\) is a real linear subspace of
\(\Mat_d(\C)^2\).

Our treatment of the closures \(\mathcal C_Q\) is set-theoretic. We do
not address whether each \(L_Q\) extends across the low-rank locus as a
special Lagrangian integral current.

Write \(\Herm_d^{\geq0}\) for the cone of positive-semidefinite
Hermitian matrices.
The following proposition shows that taking the closure of \(L_Q\)
amounts precisely to allowing the common factor \(P\) in
\Cref{prop:balanced-coordinates} to become positive semidefinite.

\begin{proposition}
  \label[proposition]{prop:leaf-closures}
  Fix \(Q\in\U_d\). A point
  \(\mathbf W\in\Mat_d(\C)^N\) belongs to
  \(\mathcal C_Q\) if and only if there exist
  \(P\in\Herm_d^{\geq0}\) and unitary matrices
  \(V_0,\ldots,V_N\in\U_d\), with
  \(V_0=\Id_d\) and \(V_N=Q\), such that
  \begin{equation}
    W_k
    =
    V_kPV_{k-1}^*,
    \qquad
    1\leq k\leq N.
    \label{eq:semidefinite-common-factorization}
  \end{equation}
\end{proposition}

\begin{proof}
  Suppose first that \(\mathbf W\) admits a factorization
  \eqref{eq:semidefinite-common-factorization} with
  \(P\in\Herm_d^{\geq0}\), \(V_0=\Id_d\), and \(V_N=Q\).
  For \(\varepsilon>0\), set
  \[
    P_\varepsilon
    :=
    P+\varepsilon\Id_d
    \in\Herm_d^+
  \]
  and define
  \[
    W_k^{(\varepsilon)}
    :=
    V_kP_\varepsilon V_{k-1}^*,
    \qquad
    1\leq k\leq N.
  \]
  Then
  \(
    \mathbf W^{(\varepsilon)}
    :=
    (W_1^{(\varepsilon)},\ldots,W_N^{(\varepsilon)})
  \)
  is balanced and has invertible components. Moreover, its product is
  \[
    W_N^{(\varepsilon)}
    \cdots
    W_1^{(\varepsilon)}
    =
    QP_\varepsilon^N
    \in
    Q\Herm_d^+.
  \]
  Hence
  \(
    \mathbf W^{(\varepsilon)}\in L_Q.
  \)
  Since
  \(
    \mathbf W^{(\varepsilon)}\to\mathbf W
  \)
  as \(\varepsilon\to0\), it follows that
  \(
    \mathbf W\in\mathcal C_Q.
  \)

  Conversely, let \(\mathbf W\in\mathcal C_Q\). Choose a sequence
  \(
    \mathbf W^{(j)}\in L_Q
  \)
  converging to \(\mathbf W\). By
  \Cref{prop:balanced-coordinates}, for every \(j\) there exist
  \[
    P_j\in\Herm_d^+,
    \qquad
    V_0^{(j)},\ldots,V_N^{(j)}\in\U_d,
  \]
  with
  \(
    V_0^{(j)}=\Id_d
  \)
  and
  \(
    V_N^{(j)}=Q,
  \)
  such that
  \begin{equation}
    W_k^{(j)}
    =
    V_k^{(j)}P_j
    \bigl(V_{k-1}^{(j)}\bigr)^*,
    \qquad
    1\leq k\leq N.
    \label{eq:approximating-positive-factorizations}
  \end{equation}
  Since \(V_0^{(j)}=\Id_d\), uniqueness of the factorization gives
  \[
    P_j
    =
    \left(
      \bigl(W_1^{(j)}\bigr)^*
      W_1^{(j)}
    \right)^{1/2}.
  \]
  Therefore
  \[
    P_j
    \longrightarrow
    P
    :=
    (W_1^*W_1)^{1/2}
    \in\Herm_d^{\geq0}.
  \]

  The group \(\U_d^N\) is compact. After passing to a subsequence, we
  may therefore assume that
  \[
    V_k^{(j)}
    \longrightarrow
    V_k\in\U_d,
    \qquad
    1\leq k\leq N.
  \]
  We have \(V_0=\Id_d\) and \(V_N=Q\). Passing to the limit in
  \eqref{eq:approximating-positive-factorizations} gives
  \[
    W_k
    =
    V_kPV_{k-1}^*,
    \qquad
    1\leq k\leq N.
  \]
  Thus \(\mathbf W\) has a semidefinite common factorization with
  terminal unitary factor \(Q\).
\end{proof}

We next extend the existence part of
\Cref{prop:balanced-coordinates} to arbitrary points of the balanced
variety.

\begin{lemma}
  \label[lemma]{lem:semidefinite-common-factorization}
  For every
  \(
    \mathbf W=(W_1,\ldots,W_N)\in\mathcal M_{\mathbf0},
  \)
  there exist \(P\in\Herm_d^{\geq0}\) and unitary matrices
  \(V_0,\ldots,V_N\in\U_d\), with \(V_0=\Id_d\), such that
  \[
    W_k
    =
    V_kPV_{k-1}^*,
    \qquad
    1\leq k\leq N.
  \]
\end{lemma}

\begin{proof}
  For each \(1\leq k\leq N\), write the polar decomposition
  \[
    W_k
    =
    U_kH_k,
    \qquad
    H_k
    :=
    (W_k^*W_k)^{1/2}
    \in\Herm_d^{\geq0},
  \]
  where \(U_k\) is the partial isometry with initial space
  \((\ker W_k)^\perp\) and final space \(\operatorname{im}W_k\).
  Extend
  \(
    U_k|_{(\ker W_k)^\perp}
  \)
  to a unitary matrix \(A_k\in\U_d\). Since
  \[
    \operatorname{im}H_k
    =
    (\ker W_k)^\perp,
  \]
  we still have
  \(
    W_k=A_kH_k.
  \)

  The equations defining \(\mathcal M_{\mathbf0}\) give
  \[
    H_{k+1}^2
    =
    W_{k+1}^*W_{k+1}
    =
    W_kW_k^*
    =
    A_kH_k^2A_k^*,
    \qquad
    1\leq k<N.
  \]
  Both \(H_{k+1}\) and \(A_kH_kA_k^*\) are positive semidefinite.
  Uniqueness of the positive-semidefinite square root therefore gives
  \[
    H_{k+1}
    =
    A_kH_kA_k^*,
    \qquad
    1\leq k<N.
  \]

  Set
  \[
    V_0
    :=
    \Id_d,
    \qquad
    P
    :=
    H_1,
    \qquad
    V_k
    :=
    A_kV_{k-1},
    \quad
    1\leq k\leq N.
  \]
  The same induction as in the proof of
  \Cref{prop:balanced-coordinates} gives
  \[
    H_k
    =
    V_{k-1}PV_{k-1}^*,
    \qquad
    1\leq k\leq N.
  \]
  Consequently,
  \[
    W_k
    =
    A_kH_k
    =
    V_kPV_{k-1}^*,
    \qquad
    1\leq k\leq N.
  \]
\end{proof}

Thus for every \(Q\in\U_d\),
  \begin{equation}
    \mathcal C_Q
    =
    \left\{
      \mathbf W\in\mathcal M_{\mathbf0}
      \ \middle|\
      W_N\cdots W_1\in Q\Herm_d^{\geq0}
    \right\}
    =
    \mathcal M_{\mathbf0}
    \cap
    \phi^{-1}\bigl(Q\Herm_d^{\geq0}\bigr).
  \end{equation}

Since \(\mathcal M_{\mathbf0}\) is closed and contains \(L_Q\),
each closure \(\mathcal C_Q\) is contained in
\(\mathcal M_{\mathbf0}\). Conversely, every
point \(\mathbf W\in\mathcal M_{\mathbf 0}\) belongs to
\(\mathcal C_Q\) for some
\(Q\in\U_d\). Hence
\begin{equation}
  \mathcal M_{\mathbf0}
  =
  \bigcup_{Q\in\U_d}\mathcal C_Q.
  \label{eq:balanced-variety-as-union-of-closures}
\end{equation}
In particular, the family \(\{\mathcal C_Q\}_{Q\in\U_d}\) extends the foliation
of \(\mathcal M\) by the leaves \(L_Q\) to a cover of the whole
balanced variety.

The structure of this cover is governed by the rank of the common
factor \(P\). Since multiplication by unitary matrices preserves rank,
every \(\mathbf W\in\mathcal C_Q\) satisfies
\[
  \operatorname{rank}W_1
  =\cdots
  =\operatorname{rank}W_N
  =\operatorname{rank}P.
\]
Consequently,
\(
  \mathcal C_Q\cap\GL_d(\C)^N=L_Q,
\)
and \(\mathcal C_Q\setminus L_Q\) consists exactly of the tuples whose
components are rank deficient.

When \(d=1\), the only possible rank drop forces \(P=0\). Therefore
\(
  \mathcal C_Q=L_Q\cup\{\mathbf 0\}.
\)
In particular,
\(
  \mathcal C_{\Id_1}=C_{\mathrm{HL}}.
\)

When \(d>1\), a rank-deficient common factor \(P\) need not be zero.
Indeed, fix \(1\leq r<d\) and set
\[
  P_r
  :=
  \operatorname{diag}
  \bigl(
    \underbrace{1,\ldots,1}_{r},
    0,\ldots,0
  \bigr).
\]
Taking
\[
  V_0=\cdots=V_{N-1}=\Id_d,
  \qquad
  V_N=Q,
\]
in \Cref{prop:leaf-closures} gives
\[
  \left(
    \underbrace{P_r,\ldots,P_r}_{N-1},
    QP_r
  \right)
  \in
  \mathcal C_Q\setminus L_Q.
\]
Every component of this tuple has rank \(r\).

On \(\mathcal M\), the product is invertible, so its unitary
polar factor is unique and the fibers \(L_Q\) are pairwise disjoint.
At rank-deficient points of
\(\mathcal M_{\mathbf0}\), the unitary polar factor of
the product need not be unique. Consequently, the closures
\(\mathcal C_Q\) may overlap. The
following proposition shows that every point of
\(\mathcal C_Q\setminus L_Q\) belongs to another closure.

\begin{proposition}
  \label[proposition]{prop:overlapping-closures}
  Fix \(Q\in\U_d\). Every point of
  \(\mathcal C_Q\setminus L_Q\) belongs to
  \(\mathcal C_{Q'}\) for some \(Q'\in\U_d\) with \(Q'\neq Q\).
\end{proposition}

\begin{proof}
  Let
  \(
    \mathbf W\in\mathcal C_Q\setminus L_Q.
  \)
  By \Cref{prop:leaf-closures}, there is a factorization
  \[
    W_k
    =
    V_kPV_{k-1}^*,
    \qquad
    V_0=\Id_d,
    \qquad
    V_N=Q,
  \]
  with \(P\in\Herm_d^{\geq0}\). The matrix \(P\) is singular. Indeed, if \(P\in\Herm_d^+\), then all
the matrices \(W_k\) would be invertible and
\[
  W_N\cdots W_1
  =
  QP^N
  \in
  Q\Herm_d^+,
\]
which would imply
\(
  \mathbf W\in L_Q.
\)

  Since \(P\) is Hermitian,
  \[
    \C^d
    =
    \operatorname{im}P
    \oplus
    \ker P.
  \]
  Choose a nonidentity unitary map \(S_0\) on \(\ker P\), and define
  \[
    S
    :=
    \Id_{\operatorname{im}P}
    \oplus
    S_0
    \in\U_d.
  \]
  Then \(S\neq\Id_d\) and
  \(
    SP=P.
  \)
  Set
  \(
    Q':=QS.
  \)
  Since \(S\neq\Id_d\), we have \(Q'\neq Q\), while
  \[
    Q'PV_{N-1}^*
    =
    QSPV_{N-1}^*
    =
    QPV_{N-1}^*
    =
    W_N.
  \]
  Thus replacing the terminal factor \(V_N=Q\) by \(V_N=Q'\) does
  not change the tuple \(\mathbf W\). By
  \Cref{prop:leaf-closures},
  \(
    \mathbf W\in\mathcal C_{Q'}.
  \)
\end{proof}

\begin{remark}
  Since \(L_Q\) is invariant under positive scaling, \(\mathcal C_Q\) is
invariant under nonnegative scaling. The behavior of \(\mathcal C_Q\) under negative scaling depends on
  the parity of \(N\). For \(t<0\),
  \[
    t\mathcal C_Q
    =
    \begin{cases}
      \mathcal C_Q, & N \text{ even},\\
      \mathcal C_{-Q}, & N \text{ odd}.
    \end{cases}
  \]
  Hence \(\mathcal C_Q\) is ruled by real lines through the origin when
  \(N\) is even, while
  \(\mathcal C_Q\cup\mathcal C_{-Q}\) is ruled by such lines when
  \(N\) is odd.
\end{remark}

\section*{Acknowledgements}
TK is grateful to Gopala Krishna Srinivasan for enlightening discussions on Lie theory and symplectic geometry.

\section*{Use of AI}
The authors used ChatGPT and Gemini as interactive tools for mathematical
discussion and editorial assistance during the preparation of this manuscript.
Any suggestions provided by these tools were checked by the authors, who take
full responsibility for the manuscript.

\printbibliography

@InProceedings{menon_2025,
author="Menon, Govind",
title="The Geometry of the Deep Linear Network",
booktitle="XIV Symposium on Probability and Stochastic Processes",
year="2025",
pages="1--47"
}

@article{urbano_castro_2004,
    author = {Castro, Ildefonso and Urbano, Francisco},
    title = {On a new construction of special Lagrangian immersions in complex Euclidean space},
    journal = {The Quarterly Journal of Mathematics},
    volume = {55},
    number = {3},
    pages = {253-265},
    year = {2004}
}

@InProceedings{arora_cohen_hazan_2018,
  title = 	 {On the Optimization of Deep Networks: Implicit Acceleration by Overparameterization},
  author =       {Arora, Sanjeev and Cohen, Nadav and Hazan, Elad},
  booktitle = 	 {Proceedings of the 35th International Conference on Machine Learning},
  pages = 	 {244--253},
  year = 	 {2018},
  volume = 	 {80}
}

@article{MenonYu2025,
  title   = {An entropy formula for the {Deep Linear Network}},
  author  = {Menon, Govind and Yu, Tianmin},
  journal = {SIAM Journal on Mathematical Analysis},
  year    = {2026},
  note    = {To appear},
  eprint  = {2509.09088},
  archivePrefix = {arXiv},
  primaryClass  = {cs.LG}
}

@misc{lindsey_menon_2026,
      title={Regularization implies balancedness in the deep linear network}, 
      author={Kathryn Lindsey and Govind Menon},
      year={2026},
      eprint={2511.01137},
      archivePrefix={arXiv},
      primaryClass={cs.LG}
}

@phdthesis{Kotwal2026,
      title={Symmetries and Gradient Flows in the {Deep Linear Network}},
      author={Tejas Kotwal},
      year={2026},
      school={Brown University},
      type={Ph.D. dissertation},
      doi={10.26300/bd6x-0503}
}

@article{HarveyLawson1982,
  author  = {Harvey, Reese and Lawson, Jr., H. Blaine},
  title   = {Calibrated geometries},
  journal = {Acta Mathematica},
  volume  = {148},
  number  = {1},
  pages   = {47--157},
  year    = {1982}
}

@article{Joyce_2002,
author = {Dominic Joyce},
title = {{Special Lagrangian $m$-folds in $\mathbb{C}^m$ with symmetries}},
volume = {115},
journal = {Duke Mathematical Journal},
number = {1},
pages = {1 -- 51},
year = {2002}
}

@article{haskins_2004,
  author  = {Haskins, Mark},
  title   = {Special Lagrangian Cones},
  journal = {American Journal of Mathematics},
  year    = {2004},
  volume  = {126},
  number  = {4},
  pages   = {845--871}
}

@article{atiyah_hitchin_drinfeld_manin,
title = {Construction of instantons},
journal = {Physics Letters A},
volume = {65},
number = {3},
pages = {185-187},
year = {1978},
author = {M.F. Atiyah and N.J. Hitchin and V.G. Drinfeld and Yu.I. Manin}
}

@article{KronheimerNakajima1990,
  author  = {Kronheimer, Peter B. and Nakajima, Hiraku},
  title   = {{Yang--Mills} instantons on {ALE} gravitational instantons},
  journal = {Mathematische Annalen},
  year    = {1990},
  volume  = {288},
  pages   = {263--307}
}

@misc{Inauen-Menon,
  title = {Stochastic Nash evolution},
  author = {Dominik Inauen and Govind Menon},
  year = {2024},
  eprint = {2312.06541},
  archivePrefix = {arXiv},
  primaryClass = {math.PR}
}

@misc{Pfaff-Uhlenbeck,
  title = {The Continuum Limit of the Deep Linear Network},
  author = {Catherine Eva Pfaff and Karen Uhlenbeck},
  journal = {Quarterly of Applied Mathematics (to appear)},
  year = {2026},
  eprint = {2611.00503},
  archivePrefix = {arXiv},
  primaryClass = {math.DG}
}

@online{Uhlenbeck26,
  author       = {Uhlenbeck, Karen},
  title        = {A Sampling of Minimization Problems},
  year         = {2026},
  date         = {2026-01-28},
  organization = {Institute for Advanced Study},
  url          = {https://www.youtube.com/watch?v=rLCwqR6NWmM},
  note         = {Abel at IAS lecture}
}

@inproceedings{cheng1982real,
  title={The real Monge--Amp{\`e}re equation and affine flat structures},
  author={Cheng, Shiu Yuen and Yau, Shing-Tung},
  booktitle={Proceedings of the 1980 Beijing Symposium on Differential Geometry and Differential Equations},
  volume={1},
  pages = {339--370},
  year={1982}
}

@article{hildebrand2011crossratio,
  title = {The cross-ratio manifold: a model of centro-affine geometry},
  journal = {International Electronic Journal of Geometry},
  volume = {4},
  pages = {32--62},
  year = {2011},
  author = {Hildebrand, Roland},
  number = {2}
}

@article{hildebrand,
author = {Hildebrand, Roland},
title = {Canonical Barriers on Convex Cones},
journal = {Mathematics of Operations Research},
volume = {39},
number = {3},
pages = {841-850},
year = {2014}
}

@article{fox2015schwarz,
  title={A Schwarz lemma for K{\"a}hler affine metrics and the canonical potential of a proper convex cone},
  author={Fox, Daniel J.F.},
  journal={Annali di Matematica Pura ed Applicata},
  volume={194},
  number={1},
  pages={1--42},
  year={2015}
}

@article{karmarkar_1990,
author="Karmarkar, Narendra",
title="Riemannian geometry underlying interior-point methods for linear programming",
journal="Contemp. Math.",
publisher="Amer. Math. Soc.",
year="1990",
volume="114",
pages="51-75"
}

@article{bayer_lagarias_1989a,
  title = {The nonlinear geometry of linear programming. I. Affine and projective scaling trajectories},
  volume = {314},
  number = {2},
  journal = {Transactions of the American Mathematical Society},
  publisher = {American Mathematical Society (AMS)},
  author = {Bayer, D. A. and Lagarias, J. C.},
  year = {1989},
  month = Feb,
  pages = {499--526}
}

@article{menon-yu-conic,
  title={The Riemannian Langevin equation and conic programs},
  author={Govind, Menon and Tianmin, Yu},
  journal={Bulletin of the Institute of Mathematics Academia Sinica NEW SERIES},
  volume={20},
  number={3},
  year={2025},
  publisher={Bulletin of the Institute of Mathematics Academia Sinica (New series)}
}

\end{document}